\documentclass[runningheads]{llncs}
\usepackage[T1]{fontenc}
\usepackage{graphicx}
\usepackage[utf8]{inputenc}
\usepackage{amsmath}
\usepackage{multirow}
\usepackage[dvipsnames]{xcolor}
\usepackage{amssymb}
\usepackage{float}
\usepackage{tikz}
\usepackage{mathtools}
\usepackage{hyperref}
\usepackage{enumerate}
\usepackage{etaremune}
\usepackage{comment}
\usepackage{tikz}
\usetikzlibrary{decorations.markings}

\newtheorem{q}{Question}
\newtheorem{conj}{Conjecture}

\newcommand{\vpl}{$\mathbf{VPL}$}

\begin{document}

\title{Visibly Pushdown Languages in Groups}
\author{Laura Ciobanu\inst{1,2}\and Daniel Turaev\inst{1}}
\authorrunning{L. Ciobanu and D. Turaev}
\institute{Technische Universit\"at Berlin, Institut f\"ur Mathematik, 10623 Berlin, Germany
\and Heriot-Watt University, Department of Mathematics, EH14 4AS Edinburgh, UK}

\maketitle    

\begin{abstract}
In this paper we explore the connections between the class of Visibly Pushdown Languages (\vpl) and the natural sets of words one can associate to a finitely generated group. We show that the word problem of a finitely generated group is \vpl ~exactly when the group is finite.

We also show that free reduction does not preserve \vpl, and that finding solutions to equations in a free group with \vpl ~constraints (as reduced words) is undecidable. We explore the structure of sets whose full preimage is \vpl, showing these are often recognisable sets. We conjecture that, in any group, this class is precisely the recognisable sets.

\end{abstract}

Finitely generated groups are a rich and natural platform for formal languages, since all the group elements of a finitely generated group $G$ can be represented as words over the finite alphabet that is the (inverse-closed) generating set of $G$. All generating sets in this paper will be assumed to be finite.

An important connection between groups and formal languages is given by the relation between the language complexity of the \emph{word problem} of a group $G$ and the algebraic nature of $G$, where the \emph{word problem} of $G$ over a generating set $\Sigma$ (see Section \ref{sec:prelim}) is the set of all the words over $\Sigma$ that represent the trivial element in $G$. The celebrated results of Anisimov \cite{MR301981_Anisimov} and Muller-Schupp \cite{MR710250_MullerSchupp} establish that the word problem is regular exactly when a group is finite, and is context-free exactly when the group is virtually free (see Section \ref{sec:WP}). It is therefore interesting to consider classes of languages between regular and context-free, and ask whether there are groups that have word problems of complexities strictly between regular and context-free. Several classes of languages related to the context-free ones have been found to depict the word problem of certain groups (see for example \cite{Ho+2018+9+15}, \cite{SprianoKropholler_MR4000593}, \cite{Salvati_MR3354791}), but there are few results on language classes strictly contained in context-free, such as \vpl ~considered here. 

The class \vpl\ of Visibly Pushdown Languages is a class of languages situated between regular and context-free, with robust closure properties (see Proposition \ref{prop:closure}) akin to regular languages, but expressive power close to that of the context-free ones; they are the accept languages of Visibly Pushdown Automata, which were introduced in the 1980s under the name `Input Driven Pushdown Automata' \cite{MR589022_IDPDAsMehlhorn}. The first more in depth study was by Alur and Madhusudan \cite{alurmadhusudanVPLs,MR2334452_AlurMadVPL2} who came up with the name `Visibly Pushdown Language'. This class of languages is of interest in computer science because it is rich enough to help in the context of formal verification, while also having a very low complexity membership problem \cite{goller_et_al:LIPIcs.STACS.2024.38}, among other desirable properties. There are also useful algebraic and logical points of view: for example, \vpl ~forms a Boolean Algebra, and there exists a simple monadic second order theory (MSO) expressively equivalent to \vpl ~\cite{alurmadhusudanVPLs}.

Visibly Pushdown Languages also feature in recent papers on Word Equations with constraints in free monoids (see Section \ref{section:WEs}). When solving equations in free groups or monoids, where determining satisfiability is a decidable problem, imposing regular constrains on the solutions preserves decidability \cite{MR2172984_WEsplusRATgrps}. Surprisingly, solving Word Equations in free monoids with Visibly Pushdown constraints becomes undecidable \cite{DayACloserLook_MR4761457,Day_etal_expressivepowerofstrings}, even if there is only a small jump in complexity between $\mathbf{REG}$ and \vpl. Moreover, \cite{DayACloserLook_MR4761457} and \cite{Day_etal_expressivepowerofstrings} show that \vpl ~is a very expressive class of languages in the realm of Word Equations with non-rational constraints, in particular regarding length constraints. Understanding equations with \vpl ~constraints in groups is another interesting direction of research motivating this paper, especially since equations with non-rational constraints in groups have only recently begun to be studied systematically (see \cite{CEL25} and \cite{CG24}).

In this paper we show that the word problem of a group is a Visibly Pushdown Language exactly when the group is finite (Theorem \ref{prop:WordProblem}), show that a Benois-type result about closure under free reduction does not hold for \vpl~(Proposition \ref{prop:VPLBenois}), and study \vpl s as constraints on Word Equations (Proposition \ref{prop:WEsplusVPL}). We explore the structure of $\mathbf{VPL}^\forall$ sets (Section \ref{sec:vplforall}), showing these frequently coincide with recognisable sets. We further conjecture that the class $\mathbf{VPL}^\forall$ is precisely the class of recognisable sets over any group. Finally, we show that solving equations in free groups with certain \vpl~constrains is undecidable (Proposition \ref{prop:WEsplusVPL}), and suggest research directions on this topic.

\section{Preliminaries}\label{sec:prelim}

Throughout this paper, $X$ will refer to a finite generating set for a group, and $\Sigma$ will refer to the inverse closure of $X$, that is, for each element $x\in X$, there are two formal symbols $x, x^{-1} \in \Sigma,\,x\neq x^{-1}$. 
We will write it as $\Sigma = X \cup X^{-1}$.

In general, $\Sigma$ will also be the input alphabet of the automata we work with. 

We will assume some familiarity with regular and context-free languages, as well as their associated finite-state automata and pushdown automata respectively. A comprehensive introduction can be found, for example, in \cite{hopcroft_book_introtoautomata}

\subsection{VPLs and their closure properties}
\begin{definition}[Visibly Pushdown Automaton]
    A pushdown automaton on an alphabet $\Sigma$ with stack alphabet $\Gamma$is called a \emph{visibly pushdown automaton} if the alphabet $\Sigma$ can be partitioned into three components, as $\Sigma = \Sigma_c\sqcup\Sigma_i\sqcup\Sigma_r$, such that the only allowed transitions upon reading $x \in \Sigma$ are as follows:
    \begin{itemize}
        \item if $x \in \Sigma_i$, then the state transition cannot depend on, or modify, the stack.
        \item if $x \in \Sigma_c$, then the state transition again cannot depend on the stack, but it modifies the stack by adding (`pushing') a single symbol. 
        \item if $x \in \Sigma_r$, then the state transition must read the topmost symbol on the stack, and remove it (`pop' it) without pushing anything onto the stack.
    \end{itemize}

Furthermore, $\varepsilon$-transitions are not allowed. The three sets $\Sigma_c, \,\Sigma_i,\, \Sigma_r$ are the sets of ``call, internal, and response" letters, respectively. 
\end{definition}

By convention, we insist that a unique `bottom of the stack' letter $\bot \in \Gamma$ exists, and the stack is empty if only this symbol is on the stack. The symbol $\bot$ cannot be pushed onto the stack by a call letter; $\bot$ may be popped by a response letter, but if this occurs it is not removed from the stack.

\begin{remark}\label{remark:VPLinclusions}
The above criterion for an automaton to be visibly pushdown may be defined just as well on deterministic or nondeterministic automata. A priori, these need not be equivalent, but it turns out that they are, as all visibly pushdown automata are determinisable (see Theorem 2 of \cite{alurmadhusudanVPLs}).

Further, by taking a partition $\Sigma_c = \Sigma_r = \emptyset,\,\Sigma_i = \Sigma$, we see that all finite-state automata are visibly pushdown automata. 
\end{remark}

\begin{definition}[Visibly Pushdown Language]\label{def:VPL}
A language is called \emph{visibly pushdown} if it is the accept language\footnote{a VPA accepts a word if the final state when reading the word is an accepting state. Namely, there is no restrictions on the contents of the stack -- it may be non-empty.} of a visibly pushdown automaton. We denote the class of Visibly Pushdown Languages by $\mathbf{VPL}$.
\end{definition}

\begin{example}\label{ex:vpl}
\ \begin{enumerate}

    \item[(i)] 
 For $\Sigma = \{a,b\}$, partitioned into $\Sigma_c =\{a\}, \Sigma_i = \emptyset, \Sigma_r = \{b\}$, the language $\{a^nb^n: n \in \mathbb{N}\}$ is \vpl. Note that this language is not regular.
 
 \item[(ii)] The language $L$ over $\{a,b\}$ of words with equal number of $a$'s as the number of $b$'s is not \vpl ~for any partition of the alphabet.  Note that $L$ is deterministic context-free. 
 
 Indeed, suppose $V$ is a visibly pushdown automaton accepting $L$. Since $L$ is not regular, exactly one of the two letters must be a call, and the other a response. Suppose $a$ is a call, and $b$ is a response.  There are $n \neq m$ such that reading $b^n$ and $b^m$  will result in the same configuration of $V$. Since $b^na^n$ is accepted by $V$, $b^ma^n$ must also be accepted by $V$ -- a contradiction.
 \item[(iii)] Analogously, over the alphabet $\{a,a^{-1}\}$, the word problem of $\mathbb{Z}$ is not \vpl, as this is the language of words with the same number of $a$'s as $a^{-1}$'s. 
    \end{enumerate}
\end{example}

 Remark \ref{remark:VPLinclusions} and Example \ref{ex:vpl} show the following inclusion of language classes: $$\mathbf{REG}\subsetneq \mathbf{VPL}\subsetneq \mathbf{DCF}\subsetneq\mathbf{CF},$$ where $\mathbf{REG}$ is the class of regular languages, $\mathbf{DCF}$ the class of deterministic context-free languages, and $\mathbf{CF}$  the class of context-free languages.

\begin{proposition}[Theorem 1 of \cite{alurmadhusudanVPLs}]\label{prop:closure}
    The class of $\mathbf{VPL}$s over a fixed partition of $\Sigma$ is closed under the operations of union, intersection, concatenation, Kleene Star, complement, and renaming. That is to say, if $L_1, L_2$ are $\mathbf{VPL}$ over a common partition $\Sigma = \Sigma_c\sqcup\Sigma_i\sqcup\Sigma_r$, and $f: \Sigma \to \Sigma'$ is a map sending calls to calls, responses to responses, and internals to internals, then $$L_1\cup L_2, L_1\cap L_2, L_1\cdot L_2, L_1^\star, L_1^C,\text{ and }f(L_1)$$ are all Visibly Pushdown Languages.
\end{proposition}

Note that the restriction to a common partition of the alphabet is necessary, and without it the same closure properties will not all hold.
\begin{example}
    The languages $L_1 = \{a^nb^nc^*\}$ and $L_2 = \{a^\star b^nc^n\}$ are \vpl s, the first over the partition $\Sigma_c = \{a\}, \Sigma_r = \{b\}, \Sigma_i = \{c\}$, and the second over the partition $\Sigma_c = \{b\}, \Sigma_r = \{c\}, \Sigma_i = \{a\}$. The intersection $L_1\cap L_2 = \{a^nb^nc^n\}$ is not context-free, let alone \vpl.  Therefore there is no partition of $\Sigma = \{a, b, c\}$ such that both $L_1$ and $L_2$ are both visibly pushdown over this partition.
\end{example}

Similarly, the restriction that $f$ respects the partition of the alphabet is essential, as the class $\mathbf{VPL}$ is not closed under general string morphisms, nor inverse morphisms. In fact, one can associate to any context-free language $L$, over an alphabet $\Sigma$, a Visibly Pushdown Language $L'$ over a marked alphabet $\Sigma' = \Sigma_c\sqcup\Sigma_i\sqcup\Sigma_r$ with $\Sigma_c =\Sigma\times\{c\},\,\Sigma_i=\Sigma\times\{i\},\,\Sigma_r= \Sigma\times\{r\}$,
such that under the string morphism $f$ sending $(x,\alpha)\in \Sigma'$ to $x$, we get $f(L') = L$. Details of this construction can be found in \cite{alurmadhusudanVPLs}.

The class \vpl ~is further closed under a quotient operation.
\begin{definition}[Quotients of languages]
    Let $K, L$ be languages. The \emph{right quotient} of $K$ by $L$ is the language $KL^{-1}=\{u: \exists v \in L: uv \in K\}.$
    
    Analogously, the \emph{left quotient} of $K$ by $L$ is $L^{-1}K = \{u: \exists v \in L: vu \in K\}.$
\end{definition}

\begin{proposition}[Theorem 1 of \cite{VPLquotientsMR3666382}]\label{prop:quotients}
    The class $\mathbf{VPL}$ is closed under left and right quotients. That is, for $K, L \in \mathbf{VPL}$, both $K\cdot L^{-1}$ and $L^{-1}\cdot K$ are visibly pushdown languages.
\end{proposition}

Importantly, Proposition \ref{prop:quotients} holds when taking a quotient by any finite language, since these are regular, and hence \vpl ~for any partition of $\Sigma$.

\subsection{Congruences on VPLs}\label{section:congs}

The \emph{syntactic congruence} of a language $L$ is defined by $$u_1\approx_L u_2 \iff \forall v, \, w \in \Sigma ^\star: (vu_1w\in L \iff vu_2w \in L).$$ Such a congruence may be defined for any language. We say a congruence has \emph{finite index} if it has finitely many equivalence classes in $\Sigma^\star$. It is well known that the syntactic congruence of a language $L$ is finite index if and only if $L$ is regular. 
Similarly, the \emph{Myhill-Nerode congruence} is defined as $$u_1\sim_L u_2 \iff \forall v \in \Sigma ^\star: (u_1v\in L \iff u_2v \in L),$$ and also has finite index if and only if $L$ is regular. 

Congruences apply to \vpl s as well \cite{VPLcong_MR2184704}, where the Myhill-Nerode and syntactic congruences are modified according to the partition of $\Sigma$. 

\begin{definition}\label{defn:MCMRWM}\begin{enumerate}
    \item[(i)] Given a partition $\Sigma = \Sigma_c\sqcup\Sigma_i\sqcup\Sigma_r$ of an alphabet, the set of \emph{matched-response} words, denoted $\mathit{MR}(\Sigma)$, is the set of words $u$ with the property that in any prefix $u'$ of $u$, the number of response symbols is at most the number of call symbols. 
    
    \item[(ii)] Dually, the set of \emph{matched-call} words, denoted $\mathit{MC}(\Sigma)$, is the set of words $u$ with the property that in any suffix $u'$ of $u$, the number of call symbols is at most the number of response symbols. 
    \item[(iii)]The set of \emph{well-matched} words is the set $\mathit{WM}(\Sigma) = \mathit{MC}(\Sigma)\cap \mathit{MR}(\Sigma)$.
    \end{enumerate}
\end{definition}

Where $\Sigma$ is clear, we will write $\mathit{MR}, \mathit{MC},$ and $\mathit{WM}$ respectively. 

It is useful to consider languages consisting entirely of well-matched words. Adding this restriction to a \vpl\  provides a strictly smaller class of languages. We will call a language $L$ a ``well-matched \vpl'' if $L\subseteq \mathit{WM}$ and $L \in\mathbf{VPL}$.

The following three congruences are defined in \cite{VPLcong_MR2184704}.\footnote{We have used here the same notation from the original paper.} It is important to note that for the congruences $\sim_0$ and $\approx$, the congruences are not defined on all of $\Sigma^\star$.
\begin{align}
    u_1,\,u_2\in\Sigma^\star,\,&u_1\equiv u_2 \iff \forall v \in \mathit{MR}:(u_1v\in L \iff u_2 v \in L) \\
    u_1,\,u_2\in \mathit{MC}, \,&u_1 \sim_0 u_2 \iff \forall v \in\Sigma^\star:(u_1v\in L \iff u_2 v \in L) \\
    u_1,\, u_2 \in \mathit{WM}, \, &u_1 \approx u_2 \iff \forall v, w \in \Sigma^\star:(vu_1w \in L \iff vu_2 w \in L)
\end{align}
\begin{proposition}[Theorem 2 of \cite{VPLcong_MR2184704}]\label{prop:CongsVPLs}
    A language $L$ is $\mathbf{VPL}$ if and only if the congruences $\equiv,\, \sim_0,\, \approx$ are all finite index.
\end{proposition}

Note the similarity to the case of regular languages. Indeed, $\approx$ is precisely the syntactic congruence restricted to well-matched words. Likewise, $\sim_0$ and $\equiv$ are the Myhill-Nerode congruence, respectively restricted to \textit{MC} words or restricted to \textit{MR} suffixes.

\subsection{Languages in Groups}

There are many languages of interest in groups beyond the word problem. A key observation is that when representing elements by words in finitely generated monoids or groups, one often needs to decide which word(s) should represent a given element from potentially infinitely many choices.
In free monoids and groups there are canonical choices, even unique ones in the monoid case. However, in non-free structures we need to pick either certain normal forms, or certain well-defined sets: we associate to subsets of a group $G$, generated by a set $X$, languages in the free monoid $\Sigma^\star = (X\cup X^{-1})^\star$ via the canonical projection $\pi: \Sigma^\star \to G$ sending each letter $a \in \Sigma$ to $a \in X\cup X^{-1}\subset G$ and extended in the natural way to a monoid homomorphism. 

 Given a language class $\mathcal{C}$, we  associate two notions of $\mathcal{C}$-sets to a group $G$. 
 \begin{definition}
     A set $S\subseteq G$ is $\mathcal{C}^\forall$ if the full preimage $\pi^{-1}(S)\subseteq \Sigma^\star$ is in $\mathcal{C}$, and $\mathcal{C}^\exists$ if there is some $L \in \mathcal{C}$,  such that $\pi(L) = S$.
 \end{definition}

  This notation is established in \cite{carvalho2025linguisticsubsetsgroupsmonoids}, and highlights the fact that these two language classes are defined by the following formula:$$S\in \mathcal{C}^\Box \iff \exists L\in\mathcal{C}, \forall s \in S, \Box x \in \pi^{-1}(s): (\pi(L) = S \land x \in L), $$
where $\Box$ is $\forall$ and $\exists$, respectively.

When $\mathcal{C}= \mathbf{REG}$, the standard terminology in the literature is to call a  $\mathbf{REG}^\forall$ set \emph{recognisable}, and a $\mathbf{REG}^\exists$ set \emph{rational}. Various (often contradictory) terms have also been used for $\mathbf{CF}^\forall$ and $\mathbf{CF}^\exists$ sets.
The following algebraic characterisation of recognisable sets is due to Herbst and Thomas \cite[Section 6]{HERBST1993187}. \begin{proposition}[Algebraic characterisation of recognisable sets]\label{prop:recsetsareficosets} In any group $G$, the recognisable sets are finite unions of cosets of a finite index normal subgroup of $G$. That is, if  $\pi^{-1}(K)$ is regular for some $K \subseteq G$, then there exists a finite index normal subgroup $N$ such that $$K = \bigcup_{i = 1,\dots, n} N g_i:\   N\trianglelefteq_{f.i.} G, g_i \in G.$$
\end{proposition}

In the following we denote by $F(X)$ the free group generated by finite set $X$. As is standard, for any word $w$ over $X\cup X^{-1}$, $w$ is \emph{reduced} if there are no occurrences of $a a^{-1}$ in $w$, where $a \in X\cup X^{-1}$. In any free group, the set of reduced words is a regular language in $(X\cup X^{-1})^\star$.
\begin{definition}[Reduction map]
    Let $S\subseteq \Sigma^\star$ be a set, and let $\pi(S)$ be the image of $S$ in $F(X)$. The \emph{reduction map} $r$ associates to $S$ the set $r(S)\subset \Sigma^\star$ of reduced words with the same image as $S$. Or, in symbols: $$r(S) = \{w \in (X\cup X^{-1})^\star : \text{ $w$ is reduced and } \pi(w)\in \pi(S)\}.$$  
\end{definition}

\begin{proposition}\label{prop:Benois&Herbst}
\begin{itemize}
    \item[(1)] 
\label{prop:Benois}
    A set $K\subset F(X)$ is rational (i.e. $\mathbf{REG}^\exists$) if and only if $r(K)$ is regular. [Benois, \cite{benois_MR265496}]

\item[(2)] \label{prop:Herbst}
    A set $K\subset F(X)$ is $\mathbf{CF}^\forall$ if and only if $r(K)$ is context-free. [Herbst, \cite{Herbst_MR1119044}]
    \end{itemize}
\end{proposition}

\section{The Word Problem}\label{sec:WP}

For a group $G=\langle X \rangle$\footnote{We write $G = \langle X\rangle$ if $X$ is a generating set for the group $G$.} with identity (or trivial) element $1$, the \emph{word problem} of $G$, written $\mathit{WP}(G,X),$ is the set of all $w\in\Sigma^*$ such that $\pi(w)=1$ (recall that $\Sigma=X\cup X^{-1}$). In other words, $L = \pi^{-1}(1)$.

We recall that a group is \emph{virtually free} if it has a free subgroup of finite index.

\begin{proposition}
\begin{itemize}
    \item[(1)] 
    A group $G=\langle X \rangle$ has regular word problem, i.e. $\mathit{WP}(G,X)\in \mathbf{REG}$, if and only if $G$ is finite.[Anisimov, \cite{MR301981_Anisimov}]
\item[(2)]
A group $G = \langle X \rangle$ has context-free word problem, i.e. $\mathit{WP}(G,X)\in\mathbf{CF}$, if and only if $G$ is virtually free.[Muller-Schupp, \cite{MR710250_MullerSchupp}]\label{Prop:MullerSchupp}
    \end{itemize}
\end{proposition}

We show in this section that considering word problems in the class $\mathbf{VPL}$ cannot strengthen the above results. In particular, we prove the following.
\begin{theorem}\label{prop:WordProblem}
    Let $G=\langle X\rangle$ be a non-trivial, finitely generated, group and let $\pi: \Sigma^\star\to G$ be the standard projection map onto $G$. If $\mathit{WP}(G, X)\in\mathbf{VPL}$, then $\mathit{WP}(G,X)$ is regular and $G$ is finite. 
\end{theorem}

We will need the following lemma to show this result.

\begin{lemma}\label{lem:ellipticloxodromic}
    Let $G$ be virtually free, and assume $Y \subset G$ is a finite, inverse closed, generating set consisting of torsion\footnote{An element $g$ of a group is called \emph{torsion} if there is some $n\in\mathbb{N}_{n\geq 2}$ such that $g^n=1$.} elements. Then either $G$ is finite or there is some pair $x, y \in Y$ such that $xy$ is of infinite order in $G$.
\end{lemma}
The most elegant proof of the lemma relies on the fact that virtually free groups act non-trivially on simplicial trees with finite vertex and edge stabilisers, according to Bass-Serre theory. See \cite{BT16} (among many references on Bass-Serre theory) for some background and the main lemma that we require.
\begin{proof}
Suppose that $G$ is infinite and virtually free. Then $G$ acts on some infinite tree $T$ via automorphisms. If an element $g\in G$ fixes (at least) a point in the tree then we say $g$ is elliptic and write $Fix(g)$ for the set of fixed points of $g$ (see \cite[Section 3]{BT16}); otherwise $g$ is hyperbolic and it acts as a translation, with no fixed points. (Automorphisms that invert an edge may exist, but we can easily modify the tree so that the action only has elliptic and hyperbolic elements.)

Since all elements in $Y$ are torsion, they are elliptic.

An infinite, virtually free, group $G$ does not have a global fixed point because all the stabilisers of the vertices (and edges) for the action of $G$ on $T$ are finite; to see this, note that having a global fixed point would imply the existence of a vertex stabiliser which is the whole group $G$, which is infinite, a contradiction. 
As $G$ does not fix any vertex, the
intersection $ \cap _{z\in Y}Fix(z)$ of all the fixed point sets is empty. Since $Fix(z)$ is convex for any $z \in G$, there exist two elements $x,y \in Y$ such that $Fix(x)\cap Fix(y)=\emptyset$ (see \cite[page 541]{BT16}). Then by \cite[Lemma 7]{BT16} the element $xy$ is hyperbolic, so of infinite order, which proves the lemma.
\end{proof}

\begin{proof}[Proof of Theorem \ref{prop:WordProblem}]

Since $\mathit{WP}(G, X)$ is in \vpl, there is a partition of the alphabet $\Sigma$ into call, response, and internal letters. Suppose some $a\in X$ has infinite order. Now consider the language $\{a,a^{-1}\}^\star\cap \mathit{WP}(G,X)$. This is precisely the word problem of $\mathbb{Z}$ over the generating set $\{a, a^{-1}\}$. Indeed, if it wasn't, then some non-trivial word $a^k, k \in\mathbb{Z}\setminus\{0\}$ would be in $\mathit{WP}(G,X)$, and hence $a$ would have finite order. Now, since \vpl s are closed under intersections, and $\{a,a^{-1}\}^\star$ is regular, if $\mathit{WP}(G,X)$ were \vpl ~then so would $\mathit{WP}(\mathbb{Z},\{a\})$ -- a contradiction (see Example \ref{ex:vpl}(iii)). 

Therefore, for every $a \in X$ there is some $n_a\in\mathbb{N}$ such that $\pi(a^{n_a}) = 1$. So  $X$ is a finite subset of torsion elements of $G$. Since $G$ has \vpl ~word problem, and hence context-free word problem, we know by Proposition \ref{Prop:MullerSchupp} that $G$ is virtually free. Thus by Lemma \ref{lem:ellipticloxodromic}, either $G$ is finite or there are some $x, y \in X$ such that $xy$ has infinite order. Note that this also means $yx$ has infinite order, since $xy$ and $yx$ are conjugate.

Let $m$ and $n$ be the orders of $x$ and $y$ respectively, since $x$ and $y$ are torsion. Then for each $k \in \mathbb{N}$ consider the words $$(xy)^k(y^{n-1}x^{m-1})^k\,,\,(yx)^k(x^{m-1}y^{n-1})^k\in \pi^{-1}(1)=\mathit{WP}(G,X).$$

Now, if one of $x, y$ is a call, and the other a response, then one of $(xy)^k$ and $(yx)^k$ is an infinite family of well-matched ($\mathit{WM}$) words that are pairwise distinguished by suffixes -- producing infinitely many equivalence classes with respect to the congruence $\approx$ and thus contradicting Proposition \ref{prop:CongsVPLs}. If neither $x$ nor $y$ are calls, then $(xy)^k$ is an infinite family of matched-call ($\mathit{MC}$) words that are pairwise distinguished by suffixes -- again a contradiction. If $x$ and $y$ are both calls, then $(y^{n-1}x^{m-1})^k$ is a family of matched-response ($\mathit{MR}$) suffixes, and $(xy)^k$ is an infinite family pairwise distinguished by \textit{MR} suffixes -- again a contradiction. Thus we cannot have the product of two generators having infinite order, so $G$ must be finite.
\end{proof}

We recall that the \emph{co-word problem} of $G$ is the complement of the word problem of $G$ in $\Sigma^\star$, so it is the set of all $w\in\Sigma^\star$ such that $\pi(w) \neq 1$. 
\begin{corollary}
    If $G$ has \vpl ~co-word problem, then $G$ is finite.
\end{corollary}
\begin{proof}
The class    $\mathbf{VPL}$ is closed under complementation, and hence a group with \vpl ~co-word problem also has \vpl ~word problem.
\end{proof}

\begin{remark}
As is the case for regular and context-free languages, the formal language complexity of the word problem is generating set independent for the  \vpl ~ class. That is, if $\mathit{WP}(G,X)$ is \vpl ~for some generating set $X$ of $G$, it is \vpl ~for any finite generating set of $G$, so we could write `$\mathit{WP}(G)$ is \vpl'. For $\mathbf{REG}$ and $\mathbf{CF}$, this independence comes from the closure of these classes under inverse morphisms. Since $\mathbf{VPL}$ is \emph{not} closed under inverse morphisms, generating set independence is somewhat surprising. 
\end{remark}

\section{Types of \vpl ~sets in groups}

One of the most useful properties of regular languages in the context of free groups is the fact that they are closed under performing free reductions. That is, one may assume by Proposition \ref{prop:Benois&Herbst} that a rational set consists entirely of reduced words (over an inverse-closed generating set of a free group) without loss of generality. We consider here the interplay between \vpl s and free reduction. We will denote by $\mathbf{VPL}^{red}$ the class of languages of reduced words accepted by some visibly pushdown automaton.  A natural question then is whether this class coincides with $\mathbf{VPL}^\exists$ or $\mathbf{VPL}^\forall$, analogously to either Benois' or Herbst's Theorems (see Proposition \ref{prop:Benois&Herbst}). It turns out that neither case holds true.
\begin{proposition}\label{prop:VPLBenois}
    The following inclusions of language classes are strict: $$\mathbf{VPL}^\forall\subsetneq \mathbf{VPL}^{red}\subsetneq \mathbf{VPL}^\exists.$$
\end{proposition}

\begin{proof}
    We show this via two examples, one for each of the strict inclusions.
    \begin{enumerate}
        \item  Consider the group $\mathbb{Z}$ with generating set $\{a, a^{-1}\}$. The set $\{0\}\subset\mathbb{Z}$ is a rational set (represented by $\varepsilon$), hence a \vpl$^{red}$ set by Benois' Theorem, but it is not $\mathbf{VPL}^\forall$, since the preimage of $\{0\}$ is  $\mathit{WP}(\mathbb{Z},\{a, a^{-1}\}) $. 
        \item  Now consider the free group $F_2$ of rank $2$ over $\Sigma = \{a, b, a^{-1}, b^{-1}\}$ and let $L := \{(aaa^{-1})^nb^{2n}\, |\, n\in\mathbb{N}\}\subset \Sigma^\star.$ With the partition $\Sigma_c = \{a\}, \, \Sigma_i = \{a^{-1},b^{-1}\},\ \Sigma_r = \{b\}$, the language $L$ is the accept language of the VPA in Figure \ref{fig:VPAcounterexample}. Thus $\pi(L) \subset F_2$ is a\ $\mathbf{VPL}^\exists$ set. 
        
        However, the set of reduced words $r(L) = \{a^nb^{2n} \mid n\in\mathbb{N}\}$ is not \vpl\  ~for any partition of the alphabet. 
Indeed, consider the infinite family $\{a^nb^n \mid n\in\mathbb{N}\}$. Each pair of words $a^ib^i$ and $a^jb^j$ for $i\neq j$ in this family can be pairwise distinguished from one another by the suffix $b^j$. If $b$ is not a response, then $b^j\in \mathit{MR}$ for every $j$. Thus the congruence $\equiv$ is infinite index for $r(L)$. On the other hand, if $b$ is a response, then $a^nb^n\in \mathit{MC}$ for each $n$, and hence the congruence $\sim_0$ is infinite index. Either way, by Proposition \ref{prop:CongsVPLs}, $r(L)$ cannot be \vpl.
    \end{enumerate}\end{proof}

\begin{figure}

\begin{center}
\begin{tikzpicture}[scale=0.12]
\tikzstyle{every node}+=[inner sep=0pt]
\draw [black] (10.1,-12.6) circle (3);
\draw [black] (10.1,-12.6) circle (2.4);
\draw [black] (25.4,-12.6) circle (3);
\draw [black] (33.6,-26.6) circle (3);
\draw [black] (41.7,-12.6) circle (3);
\draw [black] (56.6,-12.6) circle (3);
\draw [black] (71.7,-12.6) circle (3);
\draw [black] (71.7,-12.6) circle (2.4);
\draw [black] (4.2,-12.6) -- (7.1,-12.6);
\fill [black] (7.1,-12.6) -- (6.3,-12.1) -- (6.3,-13.1);
\draw [black] (26.92,-15.19) -- (32.08,-24.01);
\fill [black] (32.08,-24.01) -- (32.11,-23.07) -- (31.25,-23.57);
\draw (28.85,-20.84) node [left] {$a,\$$};
\draw [black] (44.7,-12.6) -- (53.6,-12.6);
\fill [black] (53.6,-12.6) -- (52.8,-12.1) -- (52.8,-13.1);
\draw (49.15,-13.1) node [below] {$b,\$$};
\draw [black] (57.923,-15.28) arc (54:-234:2.25);
\draw (56.6,-19.85) node [below] {$b,\$$};
\fill [black] (55.28,-15.28) -- (54.4,-15.63) -- (55.21,-16.22);
\draw [black] (59.6,-12.6) -- (68.7,-12.6);
\fill [black] (68.7,-12.6) -- (67.9,-12.1) -- (67.9,-13.1);
\draw (64.15,-13.1) node [below] {$b,\#$};
\draw [black] (13.1,-12.6) -- (22.4,-12.6);
\fill [black] (22.4,-12.6) -- (21.6,-12.1) -- (21.6,-13.1);
\draw (17.75,-13.1) node [below] {$a,\#$};
\draw [black] (35.1,-24) -- (40.2,-15.2);
\fill [black] (40.2,-15.2) -- (39.36,-15.64) -- (40.23,-16.14);
\draw (38.3,-20.83) node [right] {$a^{-1}$};
\draw [black] (38.7,-12.6) -- (28.4,-12.6);
\fill [black] (28.4,-12.6) -- (29.2,-13.1) -- (29.2,-12.1);
\draw (33.55,-12.1) node [above] {$a,\$$};
\end{tikzpicture}
\end{center}

\caption{A visibly pushdown automaton accepting $\{(aaa^{-1})^nb^{2n}\, |\, n\in\mathbb{N}\}$. Here, $a$ is a call, and pushes the symbol after the comma onto the stack. Meanwhile $b$ is a response, and pops the symbol after the comma from the stack.} \label{fig:VPAcounterexample}
\end{figure}
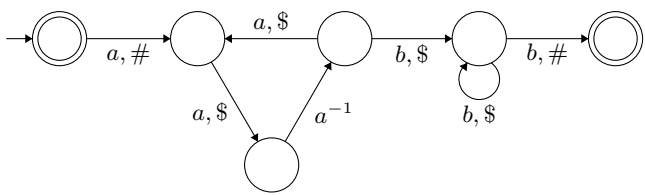

\subsection{$\mathbf{VPL}^\forall$ sets in groups}\label{sec:vplforall}

We now explore the structure of the class $\mathbf{VPL}^\forall$ over arbitrary groups. We conjecture, and make significant progress towards proving, the following.

\begin{conj}\label{conjecture1}
    In all groups, every $\mathbf{VPL}^\forall$ set is recognisable. In other words, the classes $\mathbf{VPL}^\forall$ and $\mathbf{REG}^\forall$ are equal. 
\end{conj}

We will call a partition of $\Sigma$ into calls, responses, and internals \emph{symmetric} if every call has infinite order, and  for every call $x$, $x^{-1}$ is a response, and for every response $y$, $y^{-1}$ is a call.

\begin{theorem}\label{thrm:symmetricpartitions}
    Let $G$ be a group generated by $X$. Suppose $\Sigma = \Sigma_c\sqcup\Sigma_i\sqcup\Sigma_r$ is \emph{not} a symmetric partition of $\Sigma$. Then any $\mathbf{VPL}^\forall$ subset of $G$ over this partition is a recognisable subset of $G$. 
\end{theorem}

\begin{proof}
The goal is to use the structure of the partition and the congruences outlined in Section \ref{section:congs} to show that the Myhill-Nerode congruence must have finite-index.

In order for our partition to not be symmetric, either there must exist some call $x$ such that $x^{-1}$ is not a response, or some response $y$ such that $y^{-1}$ is not a call, or some torsion element $z$ which is either a call or a response. 

Recalling the sets $\mathit{MR}$ and $\mathit{MC}$ from Definition \ref{defn:MCMRWM}, in the first of these cases, $\pi(\mathit{MR}) = G$. Indeed, for every $w \in \Sigma^\star$, consider the word $x^nx^{-n}w$, where $n = |w|$. If $w$ has any unmatched responses, they will match to a call letter from the prefix $x^nx^{-n}$, which has no response letters in it. So $x^nx^{-n}w \in \mathit{MR}$ and $\pi(x^nx^{-n}w) = \pi(w).$

Dually, in the second of the above cases, $\pi(\mathit{MC}) = G$. Indeed, for every $w \in \Sigma^\star$, consider the word $wy^ny^{-n}$, where $n = |w|$. If $w$ has any unmatched calls, they will match to a response letter in the suffix $y^ny^{-n}$, which has no call letters in it. Hence $wy^ny^{-n}\in \mathit{MC}$ and $\pi(wy^ny^{-n})= \pi(w)$.

Finally, in the third case of a call (resp. response) torsion element $z$, supposing $z^k = 1$, for each $w \in \Sigma^\star$ we can consider the word $z^{kn}w$ (resp. $wz^{kn}$), where $n = |w|$, to show that either $\pi(\mathit{MR}) = G$ or $\pi(\mathit{MC}) = G$, respectively. 

Now, suppose $S\subseteq G$ is a $\mathbf{VPL}^\forall$ set, i.e. $L:= \pi^{-1}(S) \in\mathbf{VPL}$. Consider the Myhill-Nerode congruence $\sim_L$ of the preimage language. Since $L$ is a full preimage, $u_1\sim_L u_2$ is equivalent to the statement $\forall v\in\Sigma^\star: \pi(u_1v) \in S \iff \pi(u_2v) \in S$. Notice here that, since $\pi$ is a homomorphism, if for a given $v_1$, $\pi(u_1v_1)\in S \iff \pi(u_2v_1)\in S$ holds, and some $v_2$ is such that $\pi(v_2) = \pi(v_1)$, then  $\pi(u_1v_2)\in S \iff \pi(u_2v_2)\in S$ also holds.

We therefore see that $u_1\sim_Lu_2$ is equivalent to the statement $$\forall v \in W: \pi(u_1v)\in S \iff \pi(u_2v) \in S,$$ where $W\subseteq \Sigma^\star$ is a language such that $\pi(W) = G$.  

We now treat the two cases arising from the partition not being symmetric.
\begin{itemize}
    \item if $\pi(\mathit{MR}) = G$, then taking $W = \mathit{MR}$ above, we see that $u_1\sim_L u_2$ if and only if $u_1 \equiv u_2$ (recall $\equiv$ from Section \ref{section:congs}). Since $L\in\mathbf{VPL}$, $\equiv$ is finite-index by Proposition \ref{prop:CongsVPLs}, and hence so is $\sim_L$. Thus $L$ is regular and $S$ is recognisable.
    \item if $\pi(\mathit{MC}) = G$, we first note that if $u, v$ are such that $\pi(u) = \pi(v)$, then $u\sim_Lv$. If $\pi(u)\neq \pi(v)$, then we consider $\tilde{u}, \tilde{v}\in\mathit{MC}$ such that $\pi(\tilde{u})=\pi(u)$ and $\pi(\tilde{v}) = \pi(v)$. We see that $\tilde{u}\sim_L\tilde{v}$ is equivalent to $\tilde{u}\sim_0\tilde{v}$. By transitivity, $u\sim_L v$ if and only if $\tilde{u}\sim_0\tilde{v}$. Again by Proposition \ref{prop:CongsVPLs}, $\sim_0$ is finite-index, and hence so is $\sim_L$. Thus $S$ is recognisable.
\end{itemize}
\end{proof}

\begin{proposition}
    Let $G$ be a group, and suppose that $S\subseteq G$ is a $\mathbf{VPL}^\forall$ set. Suppose further that $\pi^{-1}(S)$ is a well-matched VPL, i.e. $\pi^{-1}(S)\subseteq \mathit{WM}$. Then $S$ is recognisable.
\end{proposition}

\begin{proof}
    Suppose $\Sigma_c$ is non-empty, and suppose $a$ is a call. Then for $w \in \pi^{-1}(S)$, $wa^{-1}a \in \pi^{-1}(S)$, but this word is not well-matched since the final letter is a call. Likewise, suppose $\Sigma_r$ is non-empty, and suppose $b$ is a response. Then for $w\in\pi^{-1}(S)$, $bb^{-1}w \in \pi^{-1}(S)$, but this word is not well-matched since the first letter is a response.\footnote{Here, $a^{-1}$ and $b^{-1}$ can refer to any choice of word corresponding to the inverse of $a$ and $b$. Thus, the Proposition still holds for generating alphabets that are not inverse closed and hence may not have a letter corresponding to the inverses of $a$ resp. $b$.} Thus the only possible partition of $\Sigma$ is $\Sigma_c = \Sigma_r = \emptyset,\,\Sigma_i = \Sigma$, and hence $\pi^{-1}(S)$ will be the accept language of a finite-state automaton. 
\end{proof}

\begin{proposition}
    Let $S\subseteq G$ be a finite set. Then $S$ is $\mathbf{VPL}^\forall$ if and only if $S$ is recognisable, if and only if $G$ is finite. 
\end{proposition}
\begin{proof}
    If $G$ is finite there is nothing to prove. If $G$ is infinite, $S$ cannot be recognisable by Proposition \ref{prop:recsetsareficosets}. Thus the partition of $\Sigma$ over which $\pi^{-1}(S)$ is \vpl ~must have nonempty call and response alphabets, and by Theorem \ref{thrm:symmetricpartitions}, the partition must be symmetric.

    Now, take an element $g\in S$ with $w\in\Sigma^\star$ a reduced representative of $g$. Suppose $x \in \Sigma$ of infinite order is a call, with $x^{-1}$ a response. By Proposition \ref{prop:CongsVPLs}, the congruence $\sim_0$ must be finite index on any infinite family of $\mathit{MC}$ words, including the family $\{x^{-k},\,k\in\mathbb{N}\}$. Namely, for some infinite set $I \subset \mathbb{N}$ with $1\in I$, for each $i, j \in I, x^{-i}\sim_0x^{-j}$. Thus, since $x^{-1}xw \in \pi^{-1}(S)$, for each $i \in I$, $x^{-i}xw \in \pi^{-1}(S)$. Since $S$ is finite and $I$ is infinite, there must be some pair $i \neq j,\,i,j\in I$ such that $\pi(x^{-i}xw) = \pi(x^{-j}xw)$. Since $\pi$ is a homomorphism, this implies that for some $n,\, \pi(x^{-n})$ is the identity in $G$, contradicting the fact that $x$ has infinite order.
\end{proof}

For free groups we prove an additional result about $\mathbf{VPL}^\forall$ subgroups. 
\begin{theorem} \label{thm:sbgps}
     Let $F(X)$ be a free group over the set $X$, and $S$ a finitely generated subgroup of $F(X)$. Then $S$ is a $\mathbf{VPL}^\forall$ set if and only if it has finite index in $F$, and hence is recognisable.
\end{theorem}

 To prove Theorem \ref{thm:sbgps} we use `core' or `Stalling graphs' for finitely generated subgroups of free groups \cite[Prop. 3.8]{KM02}, which we recall below. A core graph can essentially be seen as a deterministic finite state automaton with a single accepting state (which coincides with the initial state) that can produce all the elements of a finitely generated subgroup of a free group as reduced words over the generating set of the ambient free group. 

Let $F(X)$ be a finitely generated free group  with generating set $X$, and let $\Sigma=X \cup X^{-1}$, as earlier in the paper.
Let $H$ be a finitely generated subgroup of $F(X)$. The \emph{core graph} of \(H\), denoted \(C_{H}\), is a finite, connected, graph with the following properties:
\begin{enumerate}[(i)]
	\item It is \(X\)-labelled and oriented. That is, any oriented edge $(v_1,v_2)$ has some label $a \in X$, which
	 means that the edge $(v_2, v_1)$ can be crossed and has label $a^{-1}$, but $(v_2, v_1)$ and $a^{-1}$ are not explicitly drawn.
	\item It is \emph{folded}: that is, no two edges incident to a vertex have the same \(X\)-label.
	\item It contains a distinguished vertex \(\circ\).
	\item The labels of all the closed paths (or loops) at \(\circ\) give exactly \(H\), and the graph \(C_{H}\) is minimal with this property.
\end{enumerate}
%All vertices in $C_H$ have degree $\ge 2$ (the degree is the number of all edges incident to a vertex, whether incoming or outgoing), except possibly $\circ$; and every vertex is the initial vertex (and the terminal vertex) of at most $m$ edges, labelled by pairwise different letters in $X$. 

%The following holds about core graphs:

%\begin{lemma}[{\cite[Prop.~8.3]{KM02}}]	\label{lem_deg}
%	Let $H$ be a finitely generated subgroup of a free group of rank $m>1$, and let $C_H$ be its core graph. 
%	Then $H$ has finite index if and only if $C_H$ has all vertex degrees equal to $2m$.

%\end{lemma}

\begin{proof}[Proof of Theorem \ref{thm:sbgps}]
   Let $C_S$ be the core graph for the subgroup $S$ and suppose $S$ has infinite index in $F(X)$. Then by \cite[Prop.~8.3]{KM02}, there is some vertex $v$ of $C_S$ and some letter $a$ such that no edge labelled $a$ leaves $v$. 

    Pick a loop at $\circ$ that goes through $v$. We say this path is $w_1w_2$, where $w_1$ is the label of the path from $\circ$ to $v$, and $w_2$ the label of the path from $v$ to $\circ$.
Now, consider the language $L=w_1\{a,a^{-1}\}^\star w_2 \cap \pi^{-1}(S).$
Since $S\in\mathbf{VPL}^\forall$ and \vpl s are closed under intersection with regular languages, the set $L$ must be \vpl. 

We claim that $L=\{w_1\alpha w_2 \mid \alpha \in WP(\mathbb{Z},\{a,a^{-1}\})\}$. To see this, suppose there is a $w\in L$ such that $\pi(w)=\pi(w_1a^kw_2)\in S$ for some $k>0$. Then a loop with label $w_1a^kw_2$ can be traced at $\circ$ in $C_S$. Since there is no outgoing edge labelled $a$ at vertex $v$ there is no incoming edge labelled $a^{-1}$ at $v$, so $w_1$ cannot end in an $a^{-1}$ and there is no cancellation between $w_1$ and $a^k$. Thus both $a^k w_2$ and $w_2$ label paths in $C_S$ between $v$ and $\circ$. Since $C_S$ is folded (i.e. deterministic), this can only happen if there an edge labelled $a$ starting at $v$. A similar argument holds for $k<0$. Thus $k=0$, and so $w=w_1 \alpha w_2$, where $\alpha$ is a word on $a, a^{-1}$ with exponent-sum of $a$ equal to $0$, which proves the claim.

The fact that $L$ is \vpl ~implies that the word problem of $\mathbb{Z}$ is \vpl, due to the closure of \vpl ~ under right and left quotients by a finite language (Proposition \ref{prop:quotients})-- contradicting Example \ref{ex:vpl}.
\end{proof}

\section{Word Equations with Language Constraints}\label{section:WEs}

The motivation for the present paper comes from studying equations in free groups with formal language constraints, in particular in relation to the recent results in \cite{Day_etal_expressivepowerofstrings}. We use the terminology `word equations' when working in a free monoid and simply write `equations' when working in a group. % the context applies to both monoids and groups we write `(word) equations'.

Let $\mathcal{V}$ be a finite collection of
variables and $A$ a finite set of letters, and consider  the word equation $U=V$ over $(A,\mathcal{V})$, where $U, V$ are words in the free monoid $(A \cup \mathcal{V})^*$.
% is a pair $(U,V)$ with $U,V\in (A\cup\Omega)^*$. Here, $U$ and $V$ are the two sides
%of an equation and we are looking to replace the variable occurrences in $U$
%and $V$ so that the two sides become equal.
A \emph{solution} to $U=V$  is a morphism
$\sigma\colon(\mathcal{V} \cup A)^*\to A^*$ that fixes $A$ point-wise and such that $\sigma(U)=\sigma(V)$. We will refer to the question of whether a given word equation has solutions as (\textbf{WordEqn}). This problem is solvable by the groundbreaking work of Makanin \cite{MR682490_makanin82} 
(and subsequent improvements).

A \emph{word equation with rational constraints} is a word equation $U=V$
together with a regular language $R_Y\subseteq A^*$ for each variable
$Y\in\mathcal{V}$. Then $\sigma\colon\mathcal{V}\to A^*$ is a
\emph{solution} if it is a solution to $U=V$ and also
satisfies $\sigma(Y)\in R_Y$ for each $Y\in\mathcal{V}$. We will refer to the question of whether a given equation has solutions satisfying rational constraints as (\textbf{WordEqn}, $\mathbf{REG^{\exists}}$), noting that in free monoids rational and regular sets coincide. This problem is decidable by \cite{MR2172984_WEsplusRATgrps}. One can analogously define word equations with context-free  or \vpl ~constraints in free monoids, or lift all definitions and questions to free groups, where we use the notation (\textbf{GpEqn}); background on these topics can be found in the surveys \cite{MR4786835_ciobanulangssurvey,ciobanu2023languagesgroupsequations}. 

In the recent paper \cite{Day_etal_expressivepowerofstrings} it was shown that considering language classes only slightly more complex than the regular ones yields undecidability.

\begin{proposition}[Theorem 4.4 of \cite{Day_etal_expressivepowerofstrings}]\label{prop:WE+langsMonoid} In a free monoid, Word Equations with \vpl ~constraints, and hence also with $\mathbf{DCF}$ or $\mathbf{CF}$ constraints, are undecidable.
\end{proposition}

It is therefore interesting to consider \vpl ~constraints in free groups as well. On the one hand, we see in Section \ref{sec:vplforall} that $\mathbf{VPL}^\forall$ sets behave very closely to recognisable sets and if Conjecture \ref{conjecture1} holds, it would namely imply the decidability of (\textbf{GpEqn}) with $\mathbf{VPL}^\forall$ constraints. On the other hand, in the case of $\mathbf{VPL}^{red}$ sets, that is, \vpl ~sets consisting entirely of  reduced words in the free group setting, a more straightforward analysis is possible.

\begin{proposition}\label{prop:WEsplusVPL}
    In a free group, (\textbf{GpEqn}) with $\mathbf{VPL}^{red}$ constraints, and hence also with $\mathbf{CF}^\forall$ constraints, are undecidable.
\end{proposition}

\begin{proof}
    In the free group $F(X)$ generated by $X = \{x_1,\dots, x_n\}$ ($X$ is the basic generating set and $\Sigma$ is the inverse-closed one), the language of positive words $X^\star\subset F(X)$ is a regular language of reduced words. Namely $X^\star \in \mathbf{VPL}^{red}$. That means we can encode word equations over the free monoid $X^\star$ as equations over the free group $F(X)$ with the $\mathbf{VPL}^{red}$ constraint that solutions must be in $X^\star$. Since $X^\star$ is a set of reduced words by definition, every \vpl ~in $X^\star$ is also a set of reduced words. Namely, every \vpl ~constraint in $X^\star$ can be encoded as a $\mathbf{VPL}^{red}$ constraint in $F(X)$. Therefore decidability of (\textbf{GpEqn}) with $\mathbf{VPL}^{red}$ constraints in a free group would imply decidability of (\textbf{WordEqn}) with \vpl ~constraints in the free monoid. This contradicts Proposition \ref{prop:WE+langsMonoid} above. 
    
    Since $\mathbf{VPL}^{red}$ is contained in $\mathbf{CF}^\forall$ by Proposition \ref{prop:Benois&Herbst}, (\textbf{GpEqn}) with $\mathbf{CF}^\forall$ must also be undecidable. 
\end{proof}
Proposition \ref{prop:WEsplusVPL} leads to two natural questions.
\begin{q}\label{q1}
    If Conjecture \ref{conjecture1} does not hold, and hence $\mathbf{REG}^\forall\subsetneq\mathbf{VPL}^\forall$, is (\textbf{GpEqn}) with $\mathbf{VPL}^\forall$ constraints in free groups decidable? 
\end{q}

\begin{q}\label{q2}
    If (\textbf{GpEqn}) with $\mathbf{VPL}^\forall$ constraints is decidable in free groups, is (\textbf{GpEqn}) with $\mathbf{DCF}^\forall$ constraints also decidable? 
\end{q}

\section{Conclusions}

Visibly Pushdown Languages have not been systematically studied in group theory before, so this paper explores the connections between this class of languages and the natural sets of words one can associate to a finitely generated group. The only work on this topic we are aware of is \cite{MR3523360_NestedWP}, where a non-standard version of the word problem is considered. In \cite{MR3523360_NestedWP} the author shows that the `nested word problem' for a group $G$ over a generating set $X$ is \vpl ~ if and only if $G$ is virtually free. The \emph{nested word problem} for $G$ is a language $L$ over an alphabet obtained by taking three copies $X$ such that $L$ projects onto the standard word problem $\mathit{WP}(G,X)$ (but is usually not in a bijective correspondence with $\mathit{WP}(G,X)$, nor is it the full word problem of $G$ over the three copies of $X$).

So far, the study of regular, context-free, indexed, and $n$-counter languages within groups has been fruitful for algorithmic and structural questions. We hope that this paper can serve as a starting point for studying \vpl s in this setting as well. A key future direction would be to further explore Conjecture \ref{conjecture1}, and the questions that arose at the end of Section \ref{section:WEs}. Possibly the most important questions that this paper raises are about the boundary between decidability and undecidability of Word Equations when the constraints range between rational and context-free, as stated in Questions \ref{q1} and \ref{q2}.

Many other questions can be asked for sets in finitely generated groups beyond the word problem. For example, are there groups with \vpl ~(but not regular) normal forms, or \vpl ~geodesics? Similarly, many types of conjugacy languages have been classified from the formal language point of view (\cite{CH14,CHHR16}). Do any such conjugacy languages belong to the \vpl ~class?

Decision problems to do with membership in a subgroup, or more generally a subset, are also of interest. These membership problems are a strict generalisation of the word problem, and are intricately linked to solving equations with language theoretic constraints. In particular, membership in rational ($\mathbf{REG}^\exists$) or $\mathbf{CF}^\exists$ subsets has received attention (see \cite{Lohrey_surveyRATMR3495668} for a survey on these decision problems), and it is notable that the rational subset membership problem for a free group is decidable by \cite{benois_MR265496}, while the $\mathbf{CF}^\exists$ membership problem for a free group is undecidable (Theorem 3 of \cite{LohreyMembership_MR4786650}). It is therefore natural to ask questions of subset membership with respect to \vpl\ sets.

\begin{q}\label{q:subsetmembership}
    Given a generating set $\Sigma$ for a free group $F$, partitioned into a visibly pushdown alphabet, is the $\mathbf{VPL}^\exists$ membership problem for $(F, \Sigma)$ decidable? 
\end{q}

Note that while for rational and $\mathbf{CF}^\exists$ sets the subset membership problem is generating set independent, it is likely that in the case of $\mathbf{VPL}^\exists$ there will be dependence on the choice of both generating set and partition. This is again due to the class \vpl\ not being closed under morphisms nor inverse-morphisms. We conjecture that, for certain choices of generating set and (non-trivial) partition, the membership problem is undecidable, while for other choices it may be decidable. Since decidability of the rational subset membership problem is preserved by finite extensions (see \cite{Grunschlag_phdthesis}), and $\mathbf{CF}^\exists$ subset membership problem is decidable in abelian and virtually abelian groups (see \cite{Grunschlag_phdthesis} and \cite{LohreyMembership_MR4786650}), we may also ask, if Question \ref{q:subsetmembership} is answered positively, whether it remains decidable in finite extensions (i.e. virtually free groups). One may also formulate Question \ref{q:subsetmembership} for other groups which are not free, and ask whether decidability of $\mathbf{VPL}^\exists$ membership is preserved under finite extensions in more generality.

\begin{credits}
\subsubsection{\ackname} We thank Alex Levine and Joel Day for helpful discussions. 

This work was partially supported by the Heilbronn Institute for Mathematical Sciences (through a Small Grant) and the Deutsche Forschungsgemeinschaft (DFG, German Research
Foundation) under Germany's Excellence Strategy -- The Berlin Mathematics
Research Center MATH+ (EXC-2046/1, EXC-2046/2, project ID: 390685689).

\subsubsection{\discintname} The authors have no competing interests to declare that are
relevant to the content of this article. 
\end{credits}

\bibliographystyle{splncs04}
\bibliography{refs}
\end{document}